\documentclass[12pt, reqno]{amsart}
\usepackage{amsmath, amsthm, amscd, amsfonts, amssymb, graphicx, color, mathrsfs}
\usepackage[bookmarksnumbered, colorlinks, plainpages]{hyperref}
\usepackage[all]{xy}
\usepackage{slashed}

\usepackage{soul}
\usepackage{cancel}
\usepackage{ulem}

\newtheorem{theorem}{Theorem}[section]
\newtheorem{lemma}[theorem]{Lemma}

\theoremstyle{definition}

\theoremstyle{remark}
\newtheorem{remark}[theorem]{Remark}
\numberwithin{equation}{section}

\begin{document}
\setcounter{page}{1}

\title[ $L^\infty$-$\textnormal{BMO}$ bounds for  Hermite pseudo-multipliers ]{$L^\infty$-$\textnormal{BMO}$ bounds for   pseudo-multipliers associated with the harmonic oscillator}

	\author[D. Cardona]{Duv\'an Cardona}
	\address{
		Duv\'an Cardona:
		\endgraf
		Department of Mathematics: Analysis Logic and Discrete Mathematics
		\endgraf
		Ghent University
		\endgraf
		Ghent-Belgium
		\endgraf
		{\it E-mail address} {\rm duvanc306@gmail.com}
		}



\begin{abstract} In this note we  investigate some conditions of  H\"ormander-Mihlin type in order to assure   the  $L^\infty$-${\textnormal{BMO}}$ boundedness for pseudo-multipliers  of the harmonic oscillator. The $\textnormal{H}^1$-$L^1$  continuity for Hermite multipliers  also is investigated.
{\it The final version of this paper will appear in Rev. Colombiana Mat.}

\end{abstract} \maketitle
\tableofcontents

\section{Introduction} 
The aim of this paper is to investigate the boundedness from  $L^\infty(\mathbb{R}^n)$ into $\textnormal{BMO}(\mathbb{R}^n)$  for pseudo-multipliers associated with the harmonic oscillator (see e.g. S. Thangavelu \cite{Thangavelu,Thangavelu2}). As it was observed by M. Ruzhansky in \cite{CardonaRuzhansky2018}, from the point of view of the theory of pseudo-differential operators, pseudo-multipliers would be the special case of the symbolic calculus developed in M. Ruzhansky and N. Tokmagambetov  \cite{ProfRuzM:TokN:20016,ProfRuzM:TokN:20017} (see also Remark \ref{pseudodifferntial}).
Let us consider the (Hermite operator) quantum harmonic oscillator $H:=-\Delta_x+|x|^2,$ (where $\Delta_x$ is the standard Laplacian) which extends to an unbounded self-adjoint operator on $L^2(\mathbb{R}^n).$ It is a well known fact, that the Hermite functions\footnote{Each Hermite function  $\phi_{\nu}$  has the form
$
\phi_\nu:=\Pi_{j=1}^n\phi_{\nu_j},\,\,\,$ $\phi_{\nu_j}(x_j)=(2^{\nu_j}\nu_j!\sqrt{\pi})^{-\frac{1}{2}}H_{\nu_j}(x_j)e^{-\frac{1}{2}x_j^2},
$
where $x\in\mathbb{R}^n$, $\nu\in\mathbb{N}^n_0,$ and  $H_{\nu_j}(x_j):=(-1)^{\nu_j}e^{x_j^2}\frac{d^k}{dx_{j}^k}(e^{-x_j^2}) $ denotes the Hermite polynomial of order $\nu_j.$ }   $\phi_\nu,$ $\nu\in \mathbb{N}_0^n,$ are the $L^2$-eigenfunctions of $H,$ with corresponding eigenvalues  satisfying: $H\phi_\nu=(2|\nu|+n)\phi_\nu.
$ The system $\{\phi_\nu\}_{\nu\in \mathbb{N}_0^n}$, which  is a subset of the Schwartz class $ \mathscr{S}(\mathbb{R}^n),$  provides an orthonormal basis of $L^2(\mathbb{R}^n).$ So, the spectral theorem for unbounded operators implies that
\begin{equation}
    Hf(x)=\sum_{\nu\in \mathbb{N}_0^n}(2|\nu|+n)\widehat{f}(\phi_\nu),\,\,f\in \textnormal{Dom}(H),
\end{equation} where $\widehat{f}(\phi_\nu)$ is the Fourier-Hermite transform of $f$ at $\phi_\nu,$ which is given by
\begin{equation}
    \widehat{f}(\phi_\nu)=\int\limits_{\mathbb{R}^n}f(x)\phi_\nu(x)dx.
\end{equation}  If $G\subset\mathbb{R}^n$ is  the complement of a subset of zero Lebesgue measure in $\mathbb{R}^n$, the pseudo-multiplier associated with a function $m:G\times \mathbb{N}_0^n\rightarrow \mathbb{C}$ is defined by
\begin{equation}
    Af(x)=\sum_{\nu\in \mathbb{N}_{0}^n}m(x,\nu)\widehat{f}(\phi_\nu)\phi_\nu(x),\,\,\,x\in G,\,f\in\textnormal{Dom}(A).
\end{equation} 
In this sense we say that $A$ is the pseudo-multiplier associated to the function $m,$ and that $m$ is the  symbol of $A.$ In this paper the main goal  is to give conditions on $m$ in order that $A$ can be extended to a bounded operator from $L^\infty$ to $\textnormal{BMO}.$  The problem of the boundedness of pseudo-multipliers is an interesting topic in harmonic analysis (see e.g. J. Epperson\cite{Epperson}, S. Bagchi and S. Thangavelu\cite{BagchiThangavelu},  D. Cardona and M. Ruzhansky \cite{CardonaRuzhansky2018} and references therein). The problem was initially considered for multipliers of the harmonic oscillator
\begin{equation}
    Af(x)=\sum_{\nu\in \mathbb{N}_{0}^n}m(\nu)\widehat{f}(\phi_\nu)\phi_\nu(x),\,\,f\in \textnormal{Dom}(A).\footnote{  $\textnormal{Dom}(A)=\{f\in L^2(\mathbb{R}^n) : \sum_{\nu\in \mathbb{N}_{0}^n}\vert m(\nu)\widehat{f}(\phi_\nu)\vert^2<\infty \} $ is a dense subset of $L^2(\mathbb{R}^n)$. Indeed, note that $\{\phi_\nu\}_\nu\subset \textnormal{Dom}(A),$ and consequently $L^2(\mathbb{R}^n)=\overline{\textnormal{span}(\{\phi_\nu\}_\nu)}\subset \overline{\textnormal{Dom}(A)}.$}
\end{equation} 
Indeed, an early result due to S. Thangavelu (see \cite{thangavelu0,Thangavelu}) states that if $m$ satisfies the following discrete Marcienkiewicz condition
\begin{equation}\label{thangavelucondition}
|\Delta_\nu^{\alpha} m(\nu)|\leq C_{\alpha}(1+|\nu|)^{-|\alpha|},\,\,\alpha\in\mathbb{N}^n_0,\,|\alpha|\leq [\frac{n}{2}]+1,
\end{equation}
where $\Delta_\nu$ is the usual difference operator, then the corresponding multiplier $T_m:L^p(\mathbb{R}^n)\rightarrow L^p(\mathbb{R}^n)$ extends to a bounded operator for all $1<p<\infty.$ In view of Theorem 1.1 of S. Blunck \cite{Blunk}, (see also P. Chen, E. M. Ouhabaz, A. Sikora, and  L. Yan, \cite[p. 273]{COSY}), if we restrict our attention to spectral multipliers $A=m(H),$ the boundedness on $L^p(\mathbb{R}^n),$ can be assured if $m$
satisfies the H\"ormander condition of order $s,$
\begin{equation}\label{hormandercondition}
\Vert m\Vert_{l.u.H^s}:=\sup_{r>0}\Vert m(r\cdot)\eta(|\cdot|) \Vert_{H^s(\mathbb{R}^n)}=\sup_{r>0}r^{s-\frac{n}{2}}\Vert m(\cdot)\eta(r^{-1}|\cdot|) \Vert_{H^s(\mathbb{R}^n)}<\infty,\,\, \,\,
\end{equation}
where $\eta\in \mathscr{D}(0,\infty)$ and  $s>\frac{n+1}{2},$ for all $p\in [p_0,\frac{p_0}{p_0-1}],$ for some $p_0\in (1,2).$ If $|\nu|=\nu_1+\cdots +\nu_n,$ for spectral pseudo-multipliers
\begin{equation}
      Ef(x)=\sum_{\nu\in \mathbb{N}_{0}^n}m(x,2|\nu|+n)\widehat{f}(\phi_\nu)\phi_\nu(x),\,\,f\in \textnormal{Dom}(E),
\end{equation}    under one of the following conditions
\begin{itemize}
    \item J. Epperson, \cite{Epperson}: $n=1,$ $E$ bounded on $L^2(\mathbb{R})$ and 
    \begin{equation}
        |\Delta^\gamma_\nu m(x,2\nu+1)|\leq C_\gamma(2\nu+1)^{-\gamma},\,\,\,0\leq \gamma\leq 5, 
    \end{equation}
    \item S. Bagchi and S. Thangavelu, \cite{BagchiThangavelu}: $n\geq 2,$ $E$ bounded on $L^2(\mathbb{R}^n)$ and 
    \begin{equation}
        |\Delta^\gamma_\nu m(x,2|\nu|+1)|\leq C_\gamma(2|\nu|+1)^{-\gamma},\,\,\,0\leq |\gamma|\leq n+1, 
    \end{equation} 
\end{itemize} the operator $E$ extends to an operator of weak type $(1,1).$ This means that $E:L^{1}(\mathbb{R}^n)\rightarrow L^{1,\infty}(\mathbb{R}^n)$ admits a bounded extension (we denote by   $L^{1,\infty}(\mathbb{R}^n)$ the the weak $L^1$-space\footnote{which consists of those functions $f$ such that $\Vert f \Vert_{L^{1,\infty}}=\sup_{\lambda>0}\lambda\cdot  \textnormal{meas}(\{x\in \mathbb{R}^n:|f(x)|>\lambda\})<\infty.$}). In view of the Marcinkiewicz  interpolation Theorem  it follows that $E$ extends to a  bounded linear operator on $L^p(\mathbb{R}^n),$ for all $1< p\leq 2.$

We can note that in the previous results the $L^2$-boundedness of pseudo-multipliers is assumed.  The problem of finding reasonable conditions for the $L^2$-boundedness of spectral  pseudo-multipliers, was proposed by S. Bagchi and S. Thangavelu in \cite{BagchiThangavelu}. To solve this problem,  it was considered in  \cite{CardonaRuzhansky2018}, the following H\"ormander conditions,
\begin{equation}\label{primeranorma}
\Vert m \Vert_{l.u., H^s}:=\sup_{r>0,\,y\in\mathbb{R}^n} \,r^{(s-\frac{n}{2})}\Vert  \langle x\rangle^{s}\mathscr{F}[m(y,\cdot)\psi(r^{-1} |\cdot|)](x)\Vert_{L^2({\mathbb{R}}^n_x)}<\infty,
\end{equation}
\begin{equation}\label{segundanorma}
\Vert m\Vert_{l.u., \mathcal{H}^s}:=\sup_{k>0}\sup_{y\in\mathbb{R}^n} \,{2}^{k(s-\frac{n}{2})}\Vert \langle x \rangle^s \mathscr{F}^{-1}_H[m(y,\cdot)\psi(2^{-k}|\cdot|)](x)\Vert_{L^2(\mathbb{R}^n_x)}<\infty,
\end{equation}
defined by the Fourier transform $\mathscr{F}$ and the inverse Fourier-Hermite transform $\mathscr{F}_{H}^{-1}$. More precisely, the H\"ormander condition \eqref{primeranorma} of order $s>\frac{3n}{2},$ uniformly in $y\in \mathbb{R}^n,$ or the  condition \eqref{segundanorma} for $s>\frac{3n}{2}-\frac{1}{12},$ uniformly in $y\in\mathbb{R}^n,$ guarantee the $L^2$- boundedness of the pseudo-multiplier \eqref{pseudomultiplier''}. As it was pointed out in \cite{CardonaRuzhansky2018}, in \eqref{primeranorma} we consider functions $m$ on $\mathbb{R}^n\times \mathbb{R}^n,$ but to these functions we associate a pseudo-multiplier with symbol $\{m(x,\nu)\}_{x\in\mathbb{R}^n,\nu\in\mathbb{N}_0^n}.$
On the other hand, (see Corollary 2.3 of \cite{CardonaRuzhansky2018}) if we assume the condition,
\begin{equation}\label{DiscreteconditionIntro}
|\Delta_\nu^{\alpha} m(x,\nu)|\leq C_{\alpha}(1+|\nu|)^{-|\alpha|},\,\,\alpha\in\mathbb{N}^n_0,\,|\alpha|\leq \rho,
\end{equation}for $\rho=[3n/2]+1,$ then the pseudo-multiplier in  \eqref{pseudomultiplier''} extends to a bounded operator on $L^2(\mathbb{R}^n)$, and for $\rho=2n+1$ we have its $L^p(\mathbb{R}^n)$-boundedness  for all $1<p<\infty.$ Now, we record the main theorem of \cite{CardonaRuzhansky2018}:
\begin{theorem}\label{maintheoremCardonaRuzhansky}
Let us assume that $2\leq p<\infty.$ If $A=T_m$ is a pseudo-multiplier with symbol $m$ satisfying  \eqref{primeranorma}, then  under one of the following conditions,
\begin{itemize}
\item  $n\geq 2,$  $2\leq p<\frac{2(n+3)}{n+1},$ and $s>s_{n,p}:=\frac{3n}{2}+{\frac{n-1}{2}(\frac{1}{2}-\frac{1}{p})},$
\item $n\geq 2,$ $p=\frac{2(n+3)}{n+1},$ and $s>s_{n,p}:=\frac{3n}{2}+\frac{n-1}{2(n+3)},$
\item $n\geq 2,$ $\frac{2(n+3)}{n+1}<p\leq \frac{2n}{n-2},$ and  $s>s_{n,p}:=\frac{3n}{2}{-\frac{1}{6}+\frac{2n}{3}(\frac{1}{2}-\frac{1}{p})},$
\item $n\geq 2,$ $\frac{2n}{n-2}\leq p<\infty,$ and  $s>s_{n,p}:=\frac{3n-1}{2}{+n(\frac{1}{2}-\frac{1}{p})},$
\item $n=1,$ $2\leq p<4,$ $s>s_{1,p}:=\frac{3}{2},$
\item $n=1,$ $p=4,$ $s>s_{1,4}:=2,$
\item $n=1,$ $4<p<\infty,$  $s>s_{1,p}:=\frac{4}{3}{+\frac{2}{3}(\frac{1}{2}-\frac{1}{p})},$
\end{itemize}
the operator $T_m$ extends to a bounded operator on $L^p(\mathbb{R}^n).$ For $1<p\leq 2,$ under one of the following  conditions
\begin{itemize}
\item  $n\geq 2,$  $\frac{2(n+3)}{n+5}\leq p\leq 2,$ and $s>s_{n,p}:=\frac{3n}{2}+{\frac{n-1}{2}(\frac{1}{2}-\frac{1}{p})},$
\item $n\geq 2,$ $\frac{2n}{n+2}\leq p\leq \frac{2(n+3)}{n+5},$ and  $s>s_{n,p}:=\frac{3n}{2}{-\frac{1}{6}+\frac{2n}{3}(\frac{1}{2}-\frac{1}{p})},$
\item $n\geq 2,$ $1< p\leq \frac{2n}{n+2},$ and  $s>s_{n,p}:=\frac{3n-1}{2}{+n(\frac{1}{2}-\frac{1}{p})},$
\item $n=1,$ $\frac{4}{3}\leq p<2,$ $s>s_{1,p}:=\frac{3}{2},$
\item $n=1,$ $1<p<\frac{4}{3},$  $s>s_{1,p}:=\frac{4}{3}{+\frac{2}{3}(\frac{1}{2}-\frac{1}{p})},$
\end{itemize} the operator $T_m$ extends to a bounded operator on $L^p(\mathbb{R}^n)$. However, in general:
\begin{itemize} \item for every $\frac{4}{3}<p<4$ and every $n,$ the condition  $s>\frac{3n}{2}$ implies the $L^p$-boundedness of $T_m.$  \end{itemize} If the symbol $m$ of the pseudo-multiplier $T_m$ satisfies the H\"ormander condition \eqref{segundanorma}, in order to guarantee the $L^p$-boundedness of $T_m,$  in every case above we can take $s>s_{n,p}-\frac{1}{12}.$  Moreover, the condition $s>\frac{3n}{2}-\frac{1}{12}$ implies the $L^p$-boundedness of $T_m$ for all $\frac{4}{3}<p<4.$
\end{theorem}
Now we present our main result. We will provide a version of Theorem \ref{maintheoremCardonaRuzhansky} for the critical case   $p=\infty.$
  Because, in harmonic analysis the John-Nirenberg class $\textnormal{BMO}$ (see \cite{Jhon}) is a good substitute of $L^\infty$  we will investigate the boundedness of pseudo-multipliers  from $L^\infty(\mathbb{R}^n)$ to $\textnormal{BMO}(\mathbb{R}^n)$.
\begin{theorem}\label{maintheorem} Let  $A:\mathscr{S}(\mathbb{R}^n)\rightarrow \mathscr{S}(\mathbb{R}^n)$ be a   continuous linear operator such that its symbol  $m=\{m(x,\nu)\}_{x\in G,\,\nu\in \mathbb{N}_0^n}$  $(\textnormal{see } \eqref{symbol''})$ satisfies one of the following conditions,
\begin{itemize}
    \item[(CI):] $m$ satisfies the H\"ormander-Mihlin condition \begin{equation}\label{primeranorma''}
\Vert m \Vert_{l.u., H^s}:=\sup_{r>0,\,y\in\mathbb{R}^n} \,r^{(s-\frac{n}{2})}\Vert  \langle x\rangle^{s}\mathscr{F}[m(y,\cdot)\psi(r^{-1} |\cdot|)](x)\Vert_{L^2({\mathbb{R}}^n_x)}<\infty,
\end{equation}
    where   $s>\max\{\frac{7n}{4}+\varkappa,\frac{n}{2}\},$ and $\varkappa$ is defined as in  \eqref{defivarkappa},
    \item[(CII):] $m$ satisfies the Marcinkiewicz type condition, \begin{equation}
    |\Delta_\nu^\alpha m(x,\nu)|\leq C_\alpha (1+|\nu|)^{-|\alpha|},\,\,|\alpha|\leq [7n/4-1/12]+1.
\end{equation}
\end{itemize}
Then the operator
$A=T_m$ extends to a bounded operator from $L^\infty(\mathbb{R}^n)$ into $\textnormal{BMO}(\mathbb{R}^n).$
\end{theorem} Now, we will discuss some consequences of our main result.
\begin{remark}
In relation with the results of Epperson \cite{Epperson} and Bagchi and Thangavelu \cite{BagchiThangavelu} mentioned above, Theorem \ref{maintheorem} implies that under one of the following conditions,\begin{itemize}
    \item $n=1,\,\,|\Delta_\nu^\gamma m(x,2\nu+1)|\leq C_\gamma (2\nu+1)^{-|\gamma|},\,\,0\leq \gamma\leq 2,$
    \item $n\geq 2,\,\,|\Delta_\nu^\gamma m(x,2|\nu|+n)|\leq C_\gamma (2|\nu|+n)^{-|\gamma|},\,\,0\leq |\gamma|\leq [7n/4-1/12]+1,$
\end{itemize} the spectral pseudo-multiplier
\begin{equation}Ef(x)=\sum_{\nu\in \mathbb{N}_{0}^n}m(x,2|\nu|+n)\widehat{f}(\phi_\nu)\phi_\nu(x),\,\,f\in \textnormal{Dom}(E),
\end{equation}
extends to a bounded operator from $L^\infty(\mathbb{R}^n)$ into $\textnormal{BMO}(\mathbb{R}^n).$
\end{remark}
\begin{remark} For $n=1,$ Theorem \ref{maintheoremCardonaRuzhansky} implies that the symbol inequalities
\begin{eqnarray}\label{ruz}
|\Delta_\nu^{\gamma} m(x,\nu)|\leq C_{\gamma}(1+\nu)^{-\alpha},\,\,\,0\leq \gamma\leq 2,
\end{eqnarray} are sufficient conditions for the $L^p(\mathbb{R})$-boundedness of pseudo-multipliers with $\frac{4}{3}<p<4,$ and also under the estimates
\begin{eqnarray}\label{CarRuz}
|\Delta_\nu^{\gamma} m(x,\nu)|\leq C_{\gamma}(1+\nu)^{-\alpha},\,\,\,0\leq \gamma\leq 3,
\end{eqnarray}we obtain the $L^p(\mathbb{R})$-boundedness of $T_m$ for all $p\in (1,4/3)\cup(4,\infty).$ However, we can improve the conditions on the number of derivatives imposed in \eqref{CarRuz} to discrete derivatives up to order 2 in order to assure the $L^p(\mathbb{R})$-boundedness of $T_m$ for all $4/3\leq p<\infty.$ Indeed, from Theorem \ref{maintheorem}, the hypothesis \eqref{ruz} implies the boundedness of $T_m$ from $L^\infty(\mathbb{R})$ to $\textnormal{BMO}(\mathbb{R})$ and also its $L^p(\mathbb{R})$-boundedness for $4/3\leq p<\infty,$ in view of the Stein-Fefferman interpolation theorem applied to the $L^2$-$L^2$ and $L^\infty$-$\textnormal{BMO}$ boundedness results. 
\end{remark}
\begin{remark}
Let us consider  a multiplier $T_m$ of the harmonic oscillator. Theorem \ref{maintheorem} assures that under one of the following conditions,
\begin{itemize}
    \item[(CI)':] $m$ satisfies the H\"ormander-Mihlin condition \begin{equation}\label{primeranorma'''}
\Vert m \Vert_{l.u., H^s}:=\sup_{r>0} \,r^{(s-\frac{n}{2})}\Vert  \langle x\rangle^{s}\mathscr{F}[m(\cdot)\psi(r^{-1} |\cdot|)](x)\Vert_{L^2({\mathbb{R}}^n_x)}<\infty,
\end{equation}
    where   $s>\max\{\frac{7n}{4}+\varkappa,\frac{n}{2}\},$ and $\varkappa$ is defined as in  \eqref{defivarkappa},
    \item[(CII)':] $m$ satisfies the Marcinkiewicz type condition, \begin{equation}
    |\Delta_\nu^\alpha m(\nu)|\leq C_\alpha (1+|\nu|)^{-|\alpha|},\,\,|\alpha|\leq [7n/4-1/12]+1,
\end{equation}
\end{itemize} the operator $T_m$ extends to a bounded operator from $L^\infty(\mathbb{R}^n)$ into $\textnormal{BMO}(\mathbb{R}^n).$
Moreover, the duality argument shows the boundedness of $T_m$ from $\textnormal{H}^1(\mathbb{R}^n)$ into $L^1(\mathbb{R}^n).$
\end{remark}
For certain spectral aspects and applications to PDE of the theory of pseudo-multipliers  we refer the reader to the works \cite{CardonaBarraza,Cardona,CardonaEJ,CardonaRM} and \cite{thangavelu0}. This paper is organised as follows. Section \ref{Sec2} introduces the necessary background of harmonic analysis that we will use throughout this work. Finally, in Section \ref{Sec3} we prove our main theorem.

\section{Preliminaries}\label{Sec2} 
\subsection{Pseudo-multipliers of the harmonic oscillator}
To motivate the definition of pseudo-multipliers  we will prove that these operators arise, for example, as bounded linear operators on  the Schwartz class  $\mathscr{S}(\mathbb{R}^n) .$ 
\begin{theorem}\label{pseudomutiplier''}
Let us consider the set $G:=\{z\in \mathbb{R}^n:\phi_{\nu}(z)\neq 0,\textnormal{ for all }\nu\},$ and let $A:\mathscr{S}(\mathbb{R}^n)\rightarrow \mathscr{S}(\mathbb{R}^n)$ be a continuous linear operator. Then,  the function $m:G\times \mathbb{N}^n_0\rightarrow \mathbb{C},$\footnote{The symbol $m$ is defined $a.e.$ $(x,\nu)\in \mathbb{R}^n\times \mathbb{N}_0^n.$ Indeed, note that $D=\{z:\phi_\nu(z)=0\textnormal{  for some  }\nu\}$ is a countable set, has zero measure and that $m$ is defined on $G\times \mathbb{N}_0^n,$ where $G=\mathbb{R}^n-D.$}   defined by
\begin{equation}\label{symbol''}
    m(x,\nu):=\phi_{\nu}(x)^{-1}A\phi_\nu(x),\,\,\,x\in G,
\nu\in \mathbb{N}_0^n,\end{equation} satisfies the property
\begin{equation}\label{pseudomultiplier''}
    Af(x)=\sum_{\nu\in \mathbb{N}_{0}^n}m(x,\nu)\widehat{f}(\phi_\nu)\phi_\nu(x),\,\,\,x\in G,\,f\in\mathscr{S}(\mathbb{R}^n).
\end{equation}
\end{theorem}
\begin{proof} Let us assume that $A$ is a continuous linear operator $A:\mathscr{S}(\mathbb{R}^n)\rightarrow \mathscr{S}(\mathbb{R}^n) .$ Because, for every $\nu\in\mathbb{N}_0^n,$ $\phi_\nu\in \mathscr{S}(\mathbb{R}^n)=\textnormal{Dom}(A),$  define for every $x\in G,$ and $\nu\in \mathbb{N}_0^n,$ the function
\begin{equation}
   m(x,\nu):=\phi_\nu(x)^{-1} A\phi_\nu(x).
\end{equation} Let $f\in \mathscr{S}(\mathbb{R}^n)\subset L^2(\mathbb{R}^n)$ and let us consider its Hermite series
\begin{equation}
    f=\sum_{\nu\in \mathbb{N}_0^n}\widehat{f}(\phi_\nu)\phi_\nu.
\end{equation} Because $\Vert f\Vert_{L^2(\mathbb{R}^n)}^2=\sum_{\nu}|\widehat{f}(\phi_\nu)\|^2<\infty, $ by  Simon Theorem (see Theorem 1 of B. Simon \cite{Simon}), the series 
\begin{equation}
    f_{N}=\sum_{|\nu|\leq N}\widehat{f}(\phi_\nu)\phi_\nu,\,\,N\in \mathbb{N},
\end{equation} converges to $f$
in the topology of the Schwartz class $\mathscr{S}(\mathbb{R}^n).$ Because, $A:\mathscr{S}(\mathbb{R}^n)\rightarrow \mathscr{S}(\mathbb{R}^n) ,$ is a continuous linear operator, we have that $Af_n$ converges to   $Af$ in the topology of $\mathscr{S}(\mathbb{R}^n).$ Consequently, we have proved that
\begin{equation}
    Af=\sum_{\nu\in \mathbb{N}_0^n}\widehat{f}(\phi_\nu)A\phi_\nu.
\end{equation} By observing that $m(x,\nu):=\phi_\nu(x)^{-1}A\phi_\nu(x),$ we obtain the identity,
$$  Af(x)=\sum_{\nu\in \mathbb{N}_{0}^n}m(x,\nu)\widehat{f}(\phi_\nu)\phi_\nu(x),\,\,\,x\in G,\,f\in\mathscr{S}(\mathbb{R}^n). $$ So, we end the proof.
\end{proof}
\begin{remark}\label{pseudodifferntial}
It is a well known fact that several classes of pseudo-differential operators
\begin{eqnarray}
T_\sigma f(x)=\int_{\mathbb{R}^n}e^{i2\pi x\xi}\sigma(x,\xi)\widehat{f}(\xi)d\xi,\,\,\,f\in C_{0}^\infty(\mathbb{R}^n),
\end{eqnarray}
are continuous linear operators on the Schwartz class $\mathscr{S}(\mathbb{R}^n).$ For example, if $\sigma$ is a tempered and smooth function (i.e. that $\sigma\in C^\infty(\mathbb{R}^{2n})$ satisfies  $\int|\sigma(x,\xi)|(1+|x|+|\xi|)^{-\kappa}dxd\xi<\infty $ for some $\kappa>0$) then $T_\sigma:\mathscr{S}(\mathbb{R}^n)\rightarrow \mathscr{S}(\mathbb{R}^n),$ extends to a continuous linear operator. More interesting cases arise with pseudo-differential operators with symbols $\sigma$ in the H\"ormander classes, or with more generality, in the Weyl-H\"ormander classes (see L. H\"ormander \cite{Hor1,Hor2}). From Theorem \ref{pseudomutiplier''} we have that continuous pseudo-differential operators on $ \mathscr{S}(\mathbb{R}^n) $ also can be understood as  pseudo-multipliers of the harmonic oscillator.
\end{remark}

\subsection{Functions of bounded mean oscillation $\textnormal{BMO}$.} We will consider in the following two subsection the necessary notions for introducing the $\textnormal{BMO}$ and $\textnormal{H}^1$
spaces. For this, we will follow Fefferman and Stein \cite{FeSte1972}. 
Let $f$ be a locally integrable function on $\mathbb{R}^n.$ Then $f$ is of bounded mean oscillation (abreviated as $f\in \textnormal{BMO}(\mathbb{R}^n)$), if
\begin{equation}
    \sup_{Q}\frac{1}{|Q|}\int\limits_{Q}|f(x)-f_Q|dx:=\Vert f \Vert_{*}<\infty,
\end{equation} where the supremum ranges over all finite cubes $Q$ in $\mathbb{R}^n,$ $|Q|$ is the Lebesgue measure of $Q,$ and $f_Q$  denote the mean value of $f$ over $Q,$ $f_Q=\frac{1}{|Q|}\int_Q f(x)dx.$ It is a well known fact that  $L^\infty(\mathbb{R}^n)\subset \textnormal{BMO}.$ Moreover $\ln(|x|)\in \textnormal{BMO}.$ The class of
functions of bounded mean oscillation, modulo constants, is a Banach space with the norm
$\Vert \cdot\Vert_{*}$, defined above. 
According to the John-Nirenberg inequality, $f\in \textnormal{BMO}(\mathbb{R}^n)$ if and only if, the inequality 
\begin{equation}
    |\{x\in Q: |f(x)-f_Q|>\alpha\}|\leq e^{-\frac{C_\alpha}{\Vert f \Vert_{*}}}|Q|,
\end{equation}
holds true for every $\alpha>0.$ For understanding the behaviour of a function $f\in \textnormal{BMO}(\mathbb{R}^n),$ it can be checked that
\begin{equation}\label{ast}
    \int\limits_{\mathbb{R}^n}\frac{|f(x)|}{1+|x|^{n+1}}dx<\infty.
\end{equation} Moreover, a function $f\in \textnormal{BMO}(\mathbb{R}^n),$ if and only if \eqref{ast} holds and 
\begin{equation}
    \iint\limits_{|x-x_0|<\delta;0<t<\delta}t|\nabla u(x,t)|^2dxdt\lesssim \delta^n,
\end{equation}for all $x_0\in \mathbb{R}^n$ and $\delta>0.$ Here, $u(x,t)$ is the Poisson integral of $f$ defined on $\mathbb{R}^n\times(0,\infty)$  by (see Fefferman \cite{Fefferman1971}),
\begin{equation}
    u(x,t)=\int\limits_{\mathbb{R}^n}P_t(x-y)f(y)dy,\,\,P_t(x):=\frac{c_nt}{(t^2+|x|^2)^{(n+1)/2}}.
\end{equation}

\subsection{The space $\textnormal{H}^1$}   The Hardy spaces $H^p(\mathbb{D}),$ $0<p<\infty,$ were first studied as part of complex analysis by
G. H. Hardy \cite{Hardy}. An analytic function $F$ on the disk $\mathbb{D}$ is in $H^p(\mathbb{D}),$ if 
\begin{equation}
    \sup_{0<r<1}\int\limits_{-\pi}^\pi|F(re^{i\theta})|^pd\theta<\infty.
\end{equation} For $1<p<\infty,$
we can identify $H^p(\mathbb{D}),$ with $L^p(\mathbb{T}),$ where $\mathbb{T}$ is the circle. This identification does not hold, however, for $p\leq 1.$ Unfortunately, these results cannot be extended to higher dimensions
using the theory of functions of several complex variables. So,
let us introduce the Hardy space $\textnormal{H}^1(\mathbb{R}^n).$ Let $R_{1},\cdots ,R_{n},$ be the Riesz transform on $\mathbb{R}^n,$
\begin{equation}
    R_{j}f(x)=\lim\limits_{\varepsilon\rightarrow 0}\int\limits_{|\xi|>\varepsilon}e^{i2\pi x\cdot \xi}{\xi_j}/{|\xi|}\widehat{f}(\xi),\,\,\,f\in \textnormal{Dom}(R_j),
\end{equation}where $\widehat{f}(\xi)=\int_{\mathbb{R}^n}e^{-i2\pi x\cdot \xi}f(x)dx,$ is the Fourier transform of $f$ at $\xi.$ Then, $\textnormal{H}^1(\mathbb{R}^n)$ consists of those functions $f$
on $\mathbb{R}^n,$ satisfying,
\begin{equation}
    \Vert f \Vert_{\textnormal{H}^1(\mathbb{R}^n)}:=\Vert f\Vert_{L^1(\mathbb{R}^n)}+\sum_{j=1}^n\Vert R_jf\Vert_{L^1(\mathbb{R}^n)}. 
\end{equation}
The main remark in this subsection is that the dual of $\textnormal{H}^1(\mathbb{R}^n)$ is $\textnormal{BMO}(\mathbb{R}^n)$ (see Fefferman and Stein \cite{FeSte1972}). This can be understood in the following sense:
\begin{itemize}
    \item[(a).] If $\phi\in \textnormal{BMO}(\mathbb{R}^n), $ then $\Phi: f\mapsto \int\limits_{\mathbb{R}^n}f(x)\phi(x)dx,$ admits a bounded extension on $\textnormal{H}^1(\mathbb{R}^n).$
    \item[(b).] Conversely, every continuous linear functional $\Phi$ on $\textnormal{H}^1(\mathbb{R}^n)$ arises as in $\textnormal{(a)}$ with a unique element $\phi\in \textnormal{BMO}(\mathbb{R}^n).$
\end{itemize} The norm of $\phi$ as a linear functional on $\textnormal{H}^1(\mathbb{R}^n)$ is equivalent with the $\textnormal{BMO}$ norm. Important properties of the $\textnormal{BMO}$ and the $\textnormal{H}^1$ norm are the followings,
\begin{equation}\label{BMOnormduality}
 \Vert f \Vert_{*}  =\sup_{\Vert g\Vert_{\textnormal{H}^1}=1} 
\left| \int\limits_{\mathbb{R}^n}f(x)g(x)dx\right|,\,\,\,\,\,\Vert g \Vert_{\textnormal{H}^1}  =\sup_{\Vert f\Vert_{\textnormal{BMO}}=1} 
\left| \int\limits_{\mathbb{R}^n}f(x)g(x)dx\right|.
\end{equation}
 For our further analysis we will use the following fact (see Fefferman and Stein \cite[pag. 183]{FeSte1972}): if $f\in \textnormal{H}^1(\mathbb{R}^n),$ and $\phi\in \mathscr{S}(\mathbb{R}^n)$ satisfies $\int \phi(x)dx=1,$ let us define 
\begin{equation}\label{u+}
    u^{+,f}(x):=\sup_{t>0}|\phi_{t}\ast f(x)|=\sup_{t>0}\left|\int\limits_{\mathbb{R}^n}\phi_t(x-y)f(y)dy\right|,\,\,\,\phi_t(x)=t^{-n}\phi(\frac{x}{t}).
\end{equation} Then, $u^{+,f}\in L^1(\mathbb{R}^n),$ $f(x)=\lim_{t\rightarrow 0}\phi_{t}\ast f(x),$ $a.e.x,$ and there exist positive constants $A$ and $B$ satisfying
\begin{equation}\label{bound}
    A\Vert f \Vert_{\textnormal{H}^1}\leq \Vert u^{+,f} \Vert_{L^1}\leq B\Vert f \Vert_{\textnormal{H}^1}.
\end{equation}
The duals of the $H^p(\mathbb{R}^n)$ spaces, $0<p<1$, are Lipschitz spaces. This is due to
P. Duren, B. Romberg and A. Shields \cite{Duren} on the unit circle, and
to T. Walsh \cite{Walsh} in $\mathbb{R}^n$.

\subsection{The H\"ormander-Mihlin condition for pseudo-multipliers}
As we mentioned in the introduction, if $m$ is a function on $\mathbb{R}^n,$ we say that $m$ satisfies the H\"ormander condition of order $s>0,$ if
\begin{equation}\label{hormandercondition2}
\Vert m\Vert_{l.u.H^s}:=\sup_{r>0}\Vert m(r\cdot)\eta(|\cdot|) \Vert_{H^s(\mathbb{R}^n)}=\sup_{r>0}r^{s-\frac{n}{2}}\Vert m(\cdot)\eta(r^{-1}|\cdot|) \Vert_{H^s(\mathbb{R}^n)}<\infty,
\end{equation} where $H^s(\mathbb{R}^n)$ is the usual Sobolev space of order $s.$ Indeed, we also can use the following formulation for the H\"ormander-Mihlin condition,
\begin{equation}\label{hormandercondition3}
\Vert m\Vert_{l.u.H^s}:=\sup_{j\in \mathbb{Z}}\Vert m(2^j|\cdot|)\eta(\cdot) \Vert_{H^s(\mathbb{R}^n)}=\sup_{ j\in \mathbb{Z} }2^{j(s-\frac{n}{2})}\Vert m(\cdot)\eta(2^{-j}|\cdot|) \Vert_{H^s(\mathbb{R}^n)}<\infty.
\end{equation}In particular, if we choose $\eta\in\mathscr{D} (0,\infty)$ with compact support in $[1/2,2],$ and by assuming that $m$ has support in $\{\xi:|\xi|> 2\},$ we have that $m(\cdot)\eta(2^{-j}|\cdot|)=0$ for $j\leq 0.$
So, for a such symbol $m,$ we have 
\begin{equation}\label{hormandercondition33}
\Vert m\Vert_{l.u.H^s}:=\sup_{j\geq 1}\Vert m(2^j|\cdot|)\eta(\cdot) \Vert_{H^s(\mathbb{R}^n)}=\sup_{ j\geq 1 }2^{j(s-\frac{n}{2})}\Vert m(\cdot)\eta(2^{-j}|\cdot|) \Vert_{H^s(\mathbb{R}^n)}<\infty.
\end{equation}Because we define multipliers by associating to $T_m$ the restriction of $m$ to $\mathbb{N}_0^n,$ we always can split $T_m=T_0+S_m,$ where $T_0$ has symbol supported in $\{\nu:|\nu|\leq 2\}$ and the pseudo-multiplier $S_m$ has symbol supported in $\{\nu:|\nu|>2\}.$ We will apply the H\"ormander condition to $S_m$ in order to assure its $L^\infty$-$\textnormal{BMO}$ boundedness, and later we will conclude that $T_m$ is  $L^\infty$-$\textnormal{BMO}$ bounded, by observing that the $L^\infty$-$\textnormal{BMO}$ boundedness of $T_0$ is trivial. This analysis will be developed in detail in the next section, in the context of pseudo-multipliers by employing the H\"ormander type condition
\begin{equation}\label{hormandercondition4}
\Vert m\Vert_{l.u.H^s}:=\sup_{ j\geq 1,x\in \mathbb{R}^n }2^{j(s-\frac{n}{2})}\Vert m(x,\cdot)\eta(2^{-j}|\cdot|) \Vert_{H^s(\mathbb{R}^n)}<\infty,
\end{equation}for $s$ large enough which follows from \eqref{primeranorma''}.

\section{$L^\infty$-$\textnormal{BMO}$ continuity for pseudo-multipliers}\label{Sec3}
In this section we present the proof of our main result. The main strategy in the proof of Theorem \ref{maintheorem} will be a suitable Littlewood-Paley decomposition of the symbol together with  some suitable estimates for the operator norm of pseudo-multipliers associated to each part of this decomposition. Our starting point is the following lemma. We use
the symbol $X\lesssim Y$ to denote that there exists a universal constant $C$ such that $X\leq CY.$
\begin{lemma}
Let $\phi_\nu,$ $\nu\in \mathbb{N}_{0}^n$ be a Hermite function. Then, there exists $\varkappa\leq -1/12,$ such that 
\begin{equation}\label{varkappa}
    \Vert \phi_\nu \Vert_{\textnormal{BMO}}\lesssim |\nu|^{\varkappa}.
\end{equation}
\end{lemma}\begin{proof}
By using that $L^\infty\subset \textnormal{BMO},$ we have $\Vert \phi_\nu\Vert_{\textnormal{BMO}}\lesssim \Vert \phi_\nu\Vert_{L^\infty}.$ Now, from Remark 2.5 of \cite{CardonaRuzhansky2018} we can estimate $\Vert \phi_\nu\Vert_{L^\infty}\lesssim |\nu|^{-1/12} $ which implies the desired estimate. Indeed, if
\begin{equation}\label{defivarkappa}
    \varkappa:=\inf\{\omega\in \mathbb{R}:\Vert \phi_\nu \Vert_{\textnormal{BMO}}\lesssim |\nu|^{\omega}\},
\end{equation}we have that $\varkappa\leq -1/12.$
\end{proof}

\begin{proof}[Proof of Theorem \ref{maintheorem}] We will prove that if $m$ satisfies the condition $\textnormal{(CI)},$ then $A=T_m$ can be extended to a bounded operator from $L^\infty(\mathbb{R}^n)$ to $\textnormal{BMO}(\mathbb{R}^n).$  Let us consider  the operator \begin{equation}
    \mathcal{R}:=\frac{1}{2}(H-n), 
\end{equation} where $H$ is the harmonic oscillator on  $\mathbb{R}^n,$ and let us fix a dyadic
decomposition of its spectrum: we choose a function $\psi_0\in C^{\infty}_{0}(\mathbb{R}),$  $\psi_0(\lambda)=1,$  if $|\lambda|\leq 1,$ and $\psi(\lambda)=0,$ for $|\lambda|\geq 2.$ For every $j\geq 1,$ let us define $\psi_{j}(\lambda)=\psi_{0}(2^{-j}\lambda)-\psi_{0}(2^{-j+1}\lambda).$ Then we have
\begin{eqnarray}\label{deco1}
\sum_{l\in\mathbb{N}_{0}}\psi_{l}(\lambda)=1,\,\,\, \text{for every}\,\,\, \lambda>0.
\end{eqnarray}
Let us consider $f\in L^\infty(\mathbb{R}^n).$
We will decompose the symbol $m$ as
 \begin{equation}
  m(x,\nu)=m(x,\nu)(\psi_0(|\nu|)+\psi_1(|\nu|))+\sum_{k=2}^{\infty}  m_k(x,\nu),\,\,\,\,\, m_k(x,\nu):= m(x,\nu)\cdot \psi_{k}(|\nu|).
 \end{equation}
Let us define  the sequence of  pseudo-multipliers  $T_{m(j)},\,\,j\in \mathbb{N},$  associated to every symbol $m_j,$  for $j\geq 2,$ and by $T_{0} $ the operator with symbol   $\sigma\equiv m(x,\nu)(\psi_{0}+\psi_{1}).$ Then we want to show that the operator series
\begin{equation}
T_0+S_m,\,\,S_m:=\sum_{k} T_{m(k)},
\end{equation}
satisfies,
\begin{equation}
\Vert T_m \Vert_{\mathscr{B}(L^\infty(\mathbb{R}^n),\textnormal{BMO}(\mathbb{R}^n))}  \leq \Vert T_0 \Vert_{\mathscr{B}(L^\infty(\mathbb{R}^n)),\textnormal{BMO}(\mathbb{R}^n)} +\sum_k  \Vert T_{m(k)} \Vert_{\mathscr{B}(L^\infty(\mathbb{R}^n)),\textnormal{BMO}(\mathbb{R}^n)} ,
\end{equation}where the series in the right hand side converges.
Because, $f\in L^\infty(\mathbb{R}^n)$ and for every $j,$ $T_{m(j)}$ has symbol with compact support, $T_{m(j)}:L^\infty(\mathbb{R}^n)\rightarrow L^\infty(\mathbb{R}^n)$ is bounded, and consequently $T_{m(j)}f\in L^\infty(\mathbb{R}^n)\subset \textnormal{BMO}(\mathbb{R}^n).$ Now, because $T_{m(j)}f\in \textnormal{BMO}(\mathbb{R}^n),$ we will estimate its $\textnormal{BMO}$ norm $\Vert T_{m(j)}f\Vert_{*}$.
By using that every symbol $m_k$ has variable $\nu$  supported in $\{\nu:2^{k-1}\leq |\nu|\leq 2^{k+1}\},$ we have
\begin{equation*}
   T_{m(k)}f(x) =\sum_{2^{k-1}\leq |\nu|\leq 2^{k+1}}m_k(x,\nu)\phi_\nu(x)\widehat{f}(\phi_\nu),\,\,x\in\mathbb{R}^n.
\end{equation*}
Consequently,
\begin{equation}\label{step1}
    \Vert   T_{m(k)}f\Vert_{*}\leq \sum_{2^{k-1}\leq |\nu|\leq 2^{k+1}} \Vert   m_k(\cdot,\nu)\phi_\nu(\cdot) \Vert_{*}|\widehat{f}(\phi_\nu)|.
\end{equation}
From  \eqref{BMOnormduality} and by using the Fourier inversion formula we have,
\begin{align*}
    \Vert   m_k(\cdot,\nu)\phi_\nu(\cdot) \Vert_{*} &=\sup_{\Vert \Omega\Vert_{\textnormal{H}^1}=1}\left|\int\limits_{\mathbb{R}^n}  m_k(x,\nu)\phi_\nu(x)\Omega(x)dx\right|\\
    &= \sup_{\Vert \Omega\Vert_{\textnormal{H}^1}=1}\left|\int\limits_{\mathbb{R}^n}   \int\limits_{\mathbb{R}^n}e^{i2\pi \nu\cdot \xi} \widehat{m}_k(x,\xi)d\xi\,\phi_\nu(x)\Omega(x)dx\right|\\
    &\leq\sup_{\Vert \Omega\Vert_{\textnormal{H}^1}=1}\sup_{x\in \mathbb{R}^n}\int\limits_{\mathbb{R}^n}|\widehat{m}_k(x,\xi)|d\xi\times   \int\limits_{\mathbb{R}^n}|\phi_\nu(x)||\Omega(x)|dx.
\end{align*}
By the Cauchy-Schwarz inequality, and the condition $s>n/2,$ we have
\begin{equation}
     \int\limits_{\mathbb{R}^n} |\widehat{m}_{k}(x,\xi)|d\xi\leq \left( \int\limits_{\mathbb{R}^n} \langle\xi \rangle^{2s}|\widehat{m}_{k}(x,\xi)|^2d\xi  \right)^\frac{1}{2}\left( \int\limits_{\mathbb{R}^n} \langle\xi \rangle^{-2s}d\xi  \right)^\frac{1}{2}.
\end{equation} Consequently, we claim that
\begin{eqnarray}
\int\limits_{\mathbb{R}^n} |\widehat{m}_{k}(x,\xi)|d\xi\leq C\Vert m\Vert_{l.u.H^s}\times 2^{-k(s-\frac{n}{2})}.
\end{eqnarray}
Indeed, if $\tilde\psi(\lambda):=\psi_0(\lambda)-\psi_0(2\lambda),$ then $\tilde\psi\in \mathscr{D}(\mathbb{R})$ and,
\begin{align*}
    \int\limits_{\mathbb{R}^n} |\widehat{m}_{k}(x,\xi)|d\xi &\lesssim \Vert m_k(x,\cdot) \Vert_{H^s(\mathbb{R}^n)}=\Vert m(x,\cdot)\tilde\psi(2^{-k}|\cdot|) \Vert_{H^s(\mathbb{R}^n)}\\
    &\lesssim  \Vert m(x,\cdot)\Vert_{l.u.H^s}\times 2^{-k(s-\frac{n}{2})}\leq  \Vert m \Vert_{l.u., H^s}\times 2^{-k(s-\frac{n}{2})}.
\end{align*}  So, we obtain
\begin{align*}
   \Vert   m_k(\cdot,\nu)\phi_\nu(\cdot) \Vert_{*}&\leq \Vert m \Vert_{l.u., H^s}\times 2^{-k(s-\frac{n}{2})}\times\sup_{\Vert \Omega\Vert_{\textnormal{H}^1}=1}  \int\limits_{\mathbb{R}^n}|\phi_\nu(x)||\Omega(x)|dx\\
   &= \Vert m \Vert_{l.u., H^s}\times 2^{-k(s-\frac{n}{2})}\times\sup_{\Vert \Omega\Vert_{\textnormal{H}^1}=1}  \int\limits_{\mathbb{R}^n}\textnormal{sig}(\Omega(x))|\phi_\nu(x)|\Omega(x)dx,\\
\end{align*} where $\textnormal{sig}(\Omega(x))=-1,$ if $\Omega(x)<0,$ and $\textnormal{sig}(\Omega(x))=1,$ if $\Omega(x)\geq 0.$ By the duality relation \eqref{BMOnormduality}, and by using that
\begin{equation}
   \Vert \,\, \textnormal{sig}(\Omega(x))|\phi_\nu(x)|\Vert_{\textnormal{BMO}}\,\,\leq 2  \Vert \,\, |\textnormal{sig}(\Omega(x))|\phi_\nu(x)||\,\,\Vert_{\textnormal{BMO}} =2  \Vert \,\, |\phi_\nu(x)|\,\,\Vert_{\textnormal{BMO}}, 
\end{equation} we conclude that
\begin{align*}
     \Vert   m_k(\cdot,\nu)\phi_\nu(\cdot) \Vert_{*}&\lesssim  \Vert m \Vert_{l.u., H^s} 2^{-k(s-\frac{n}{2})} \sup_{\Vert \Omega\Vert_{\textnormal{H}^1}=1}   \Vert \phi_\nu \Vert_{\textnormal{BMO}}\Vert \Omega\Vert_{\textnormal{H}^1}.
\end{align*} Returning to the estimate \eqref{step1}, we can write
\begin{align*}
    \Vert   T_{m(k)}f\Vert_{*} &\leq \sum_{2^{k-1}\leq |\nu|\leq 2^{k+1}}   \Vert m \Vert_{l.u., H^s} 2^{-k(s-\frac{n}{2})}  \Vert \phi_\nu \Vert_{\textnormal{BMO}}|\widehat{f}(\phi_\nu)|\\
    &\leq \sum_{2^{k-1}\leq |\nu|\leq 2^{k+1}}   \Vert m \Vert_{l.u., H^s} 2^{-k(s-\frac{n}{2})}  \Vert \phi_\nu \Vert_{\textnormal{BMO}}\Vert \phi_\nu\Vert_{L^1}\Vert f \Vert_{L^\infty}.
\end{align*}
Thus, the analysis above implies the following estimate for the operator norm of $T_{m(k)},$ for all $k\geq 2,$
\begin{align*}
    \Vert T_{m(k)} \Vert_{ \mathscr{B}(L^\infty(\mathbb{R}^n),\textnormal{BMO}(\mathbb{R}^n)) }
    \lesssim \sum_{2^{k-1}\leq |\nu|\leq 2^{k+1}}   \Vert m \Vert_{l.u., H^s} 2^{-k(s-\frac{n}{2})}  \Vert \phi_\nu \Vert_{\textnormal{BMO}}\Vert \phi_\nu\Vert_{L^1}.
\end{align*}
By using Lemma 2.2 of \cite{CardonaRuzhansky2018} we have $\Vert\phi_\nu \Vert_{L^1(\mathbb{R}^n)}\lesssim |\nu|^{\frac{n}{4}}.$ Additionally, the inequality \eqref{varkappa}: $$\Vert\phi_\nu \Vert_{\textnormal{BMO}}\lesssim |\nu|^{\varkappa},$$ implies that
\begin{align*}
   \Vert T_{m(k)} \Vert_{ \mathscr{B}(L^\infty(\mathbb{R}^n),\textnormal{BMO}(\mathbb{R}^n)) }
    &\lesssim \sum_{2^{k-1}\leq |\nu|\leq 2^{k+1}}  2^{k(\frac{n}{4}+\varkappa) }\times  \Vert m\Vert_{l.u.H^s}\times 2^{-k(s-\frac{n}{2})}\\
   & \asymp 2^{kn}\times2^{k(\frac{n}{4}+\varkappa)}\times  \Vert m\Vert_{l.u.H^s}\times 2^{-k(s-\frac{n}{2})}.
\end{align*}
Now, by using that $T_0$ is a pseudo-multiplier whose symbol has compact support in the $\nu$-variables, we conclude that $T_0$ is bounded from $L^\infty(\mathbb{R}^n)$ to $\textnormal{BMO}(\mathbb{R}^n)$ and  $$\Vert T_0 \Vert_{\mathscr{B}(L^\infty(\mathbb{R}^n)),\textnormal{BMO}(\mathbb{R}^n)}\leq C\Vert m\Vert_{L^\infty}.$$ This analysis, allows us to estimate, the operator norm of $T_m$ as follows,
\begin{align*}
\Vert T_m \Vert_{\mathscr{B}(L^\infty(\mathbb{R}^n),\textnormal{BMO}(\mathbb{R}^n))}  &\leq \Vert T_0 \Vert_{\mathscr{B}(L^\infty(\mathbb{R}^n)),\textnormal{BMO}(\mathbb{R}^n)} +\sum_k  \Vert T_{m(k)} \Vert_{\mathscr{B}(L^\infty(\mathbb{R}^n)),\textnormal{BMO}(\mathbb{R}^n)} \\
&\lesssim \Vert m\Vert_{L^\infty}+\sum_{k=1}^{\infty}2^{-k(s-\frac{7n}{4} -\varkappa)}  \Vert m\Vert_{l.u.H^s}\\
&\leq C(\Vert m\Vert_{L^\infty}+\Vert m\Vert_{l.u.H^s})<\infty,
\end{align*}
provided that $s>\frac{7n}{4} +\varkappa,$ for some $\varkappa\leq -1/12.$ So, we have proved the $L^\infty$-$\textnormal{BMO}$ boundedness of $T_m.$ In order to end the proof we only need to prove that, under the condition $\textnormal{(CII)},$ the operator $T_m$ is bounded from $L^\infty(\mathbb{R}^n)$ to $\textnormal{BMO}(\mathbb{R}^n).$ But, if $m$
satisfies $\textnormal{(CII)},$ then it also does to satisfy $\textnormal{(CI)},$ in view of  the inequality,
\begin{equation}
    \Vert m \Vert_{l.u. H^s} \lesssim \sup_{|\alpha|\leq [7n/4-1/12]+1}(1+|\nu|)^{|\alpha|}\sup_{x,\nu}|\Delta^\alpha m(x,\nu)|,
\end{equation} for $s>0$ satisfying, 
$\frac{7n}{4} -\frac{1}{12}<s<[7n/4-1/12]+1,$ (see Eq. (2.29) of \cite{CardonaRuzhansky2018}).

\begin{remark}According to the proof of Theorem \ref{maintheorem}, 
if $T_m$ satisfies the condition $\textnormal{(CI)},$ then we have 
\begin{equation}
\Vert T_m\Vert_{\mathscr{B}(L^\infty(\mathbb{R}^n),\textnormal{BMO}(\mathbb{R}^n))}\leq C(\Vert m\Vert_{L^\infty}+\Vert m\Vert_{l.u.H^s}).
\end{equation} On the other hand, if we assume \textnormal{(CII)}, the operator norm of $T_m$ satisfies
\begin{equation}
\Vert T_m\Vert_{\mathscr{B}(L^\infty(\mathbb{R}^n),\textnormal{BMO}(\mathbb{R}^n))}\leq C\sup_{|\alpha|\leq [7n/4-1/12]+1  }(1+|\nu|)^{|\alpha|}\sup_{x,\nu}|\Delta^\alpha m(x,\nu)|.
\end{equation}
\end{remark}
\end{proof}
\noindent {\bf{Acknowledgements. }} I would like to thanks Professor Michael  Ruzhansky for several discussions on the subject.

\bibliographystyle{amsplain}

\end{document}